\documentclass[a4paper]{article}

\usepackage{amsmath,amsfonts,amssymb,amscd,amsthm,amsbsy}
\usepackage[latin1]{inputenc}   

\usepackage[a4paper]{geometry}
\usepackage{tikz}

\usepackage{pgf,tikz}
\usetikzlibrary{arrows}

\def\ds{\displaystyle}

\def \to{\rightarrow}
\def \states{\mathbb{T}^d}
\def \T{\mathbb{T}}
\def \R {\mathbb{R}}

\def \Z {\mathbb{Z}}
\def\dive {{\rm div}}
\def \inte{\int_{\T^d}}
\def \mes{\mathcal{P}}
\def\Pk{\mes(\states)}
\def\Pks{\mes_s(\states)}

\def\dk{{\bf d}_1}

\def \ep{\varepsilon}

\newcommand{\be}{\begin{equation}}
\newcommand{\ee}{\end{equation}}

\newtheorem{Theorem}{Theorem}[section]
\newtheorem{Definition}[Theorem]{Definition}
\newtheorem{Proposition}[Theorem]{Proposition}
\newtheorem{Lemma}[Theorem]{Lemma}

\title{Stable solutions in potential mean field game systems}
\author{Ariela Briani \thanks{LMPT, Université de Tours} \and Pierre Cardaliaguet\thanks{Universit\'e Paris-Dauphine, PSL Research University, CNRS, Ceremade, 75016 Paris, France. cardaliaguet@ceremade.dauphine.fr}}

\begin{document}

\maketitle

\begin{abstract} We introduce the notion of stable solution in mean field game theory:  they are locally isolated solutions of the mean field game system. We prove that such solutions exist in potential mean field games and are local attractors for learning procedures. 
\end{abstract}


\section{Introduction}

Mean field games (MFG) are Nash equilibrium configurations in differential games with infinitely many infinitesimal players. If the existence of such equilibria holds in general frameworks, one cannot expect their uniqueness without strong restrictions. However the multiplicity of equilibria is a real issue in terms of applications: indeed, when there are several equilibria, it is not clear how the players can coordinate in order to decide which equilibrium to play. The aim of this work is to introduce a notion of stable solutions for the MFG system: these stable solutions are Nash equilibria which have  the property of being locally isolated as well as local attractors for learning procedures. In this sense they are robust  with respect to perturbation. 

Let us recall that the terminology and main properties of Mean Field Games were introduced by Lasry and Lions  in a series of papers \cite{LL06cr1, LL06cr2, LL07mf}. At the same period Huang, Caines and Malhamé \cite{HCMieeeAC06} discussed the same concept under the name of  Nash certainty equivalence principle. One way to represent an MFG equilibrium is through the following system of partial differential equations.

\be\label{eq:MFG}
\left\{\begin{array}{l}
-\partial_t u -\Delta u +H(x,Du)=f(x,m) \; {\rm in}\; \T^d\times (t_0,T),\\
\partial_t m -\Delta m - \dive(mD_pH(x, Du))= 0 \; {\rm in}\; \T^d\times (t_0,T) \\
m(x,t_0)=m_0(x), \; u(x,T)=g(x,m(T))  \; {\rm in}\; \T^d.
\end{array}\right.
\ee
(To simplify the discussion related to boundary conditions, we work here in the torus  $\T^d:=\R^d/\Z^d$). In the above system the pair $(u,m)$ is the unknown. The map $u=u(x,t)$ can be interpreted as the value function of a (small) player, while $m(t)=m(x,t)$ is understood as the evolving probability density of the players at time $t$. Note that $u$ satisfies a backward Hamilton-Jacobi equation while $m$ solves a forward Kolmogorov equation with initial condition $m(t_0)=m_0$, where $m_0$ is a given probability density on $\T^d$. The Hamiltonian $H:\T^d\times \R^d\to \R$ is typically smooth and convex in the second variable. The coupling functions $f,g$ depend on the space variable and on the probability density. 

Existence of solution for the MFG system \eqref{eq:MFG} has been established under general assumptions by Lasry and Lions \cite{LL06cr2, LL07mf}. Uniqueness, however, only holds under restrictive conditions: either the horizon is small, or the couplings functions are monotone, see \cite{LL07mf}.

\bigskip

In this paper we are interested in MFG systems which might have several solutions. Our aim is to introduce a particular notion, the stable MFG equilibrium. The natural idea for this is to call stable a solution of the MFG system \eqref{eq:MFG} for which the associated  linearized system has only one solution, the trivial one. We first show (Proposition \ref{prop:stable}) that, with this definition, stable MFG solutions are isolated, in the sense that there is no other solution (with the same initial condition) in a neighborhood. If this result is very natural and quite expected, it illustrates well the notion. 

The two other results are more subtle and obtained in the framework of potential MFG games. Following \cite{CH}, we say that the MFG game is potential if there exists $F,G$ such that 
$$
\frac{\delta F}{\delta m}(m,x)= f(x,m) \qquad {\rm and }\qquad \frac{\delta G}{\delta m}(m,x)= g(x,m), 
$$
where the derivative (with respect to the measure $m$) in taken in the sense of \cite{CDLL} (see also subsection \ref{subseq:notation}).  Following \cite{LL07mf}, we know then that the MFG game has a ``potential". Namely, let us define the functional (where $m=m(x,t)$ is an evolving probability density and $w=w(x,t)$ is a vector field)
$$
J(m,w):=
  \int_{t_0}^T\inte L\left(x, \frac{w(x,t)}{m(x,t)}\right)m(x,t) dxdt+ \int_{t_0}^T F(m(t))dt + G(m(T))
$$
where $L$ is the convex conjugate of $H$ (see \eqref{eq.defL}) and the pair $(m,w)$ has to fulfill the continuity equation   
$$
\partial_t m-\Delta m +\dive(w)= 0 \; {\rm in }\; \T^d\times [t_0,T], \:  \: m(t_0)=m_0.
$$
Then any minimizer $(m,w)$ of $J$ corresponds to a solution to the MFG system \cite{LL07mf}, in the sense that there exists $u=u(x,t)$ such that the pair $(u,m)$ solves \eqref{eq:MFG} and $w=-mD_pH(x,Du)$. 

Our first main result says that, when the game is potential, there are many  initial measures $m_0$ starting from which there is a stable solution to the MFG system. Indeed, we show (Theorem \ref{theo:stable}) that, if $(u,m)$ is a MFG equilibrium on the time interval $[t_0,T]$ which corresponds to a minimizer of the potential $J$, then its restriction to any subinterval $[t_1,T]$ (where $t_1\in (t_0,T)$) is a stable MFG equilibrium. 

Second we show (in Theorem \ref{theo:CvLearning}) that, for potential MFG systems, stable equilibria are local attractors for a learning procedure. Here we consider the learning procedure inspired by the Fictitious Play \cite{BrownFictitiousPlay} and introduced by  the second author and S. Hadikhanloo in  \cite{CH}. Given $\mu^0=(\mu^0(t))$ an initial guess of the evolving probability density of the players, we define by induction the sequence $(u^n, m^n,\mu^n)$:
$$
\left\{ \begin{array}{l}
-\partial_t u^{n+1}-\Delta u^{n+1}+H(x, Du^{n+1})=f(x, \mu^n)\\
\partial_t m^{n+1}-\Delta m^{n+1}-\dive ( m^{n+1}D_pH(x,Du^{n+1}))=0\\
m^{n+1}(0)= m_0, \qquad u^{n+1}(T,x)= g(x,\mu^n(T))
\end{array}\right.
$$
and
\be\label{rule}
\mu^{n+1}= \frac{n}{n+1} \mu^n+ \frac{1}{n+1} m^{n+1}.
\ee
The interpretation of this system is that the game is played over and over. At stage $n+1$ (the $(n+1)-$th time the game is played), all the players play as if the population density is going to evolve according to $\mu^n$. The map $u^{n+1}$ is the value function of each small player. When the players play optimally in this game, the population density actually evolves according to $m^{n+1}$. The players then update their estimate of the evolving population density by taking the average of the previous observations (i.e., by the rule \eqref{rule}). 
It is proved in \cite{CH} that this simple learning procedure converges: namely, if the game is potential, then any converging subsequence of the relatively compact (for the uniform convergence) sequence $(u^n,m^n)$ is a MFG Nash equilibrium. Here we show that, if in addition the equilibrium $(u,m)$ is stable and if $\mu^0$ is sufficiently close to $m$, then the entire sequence $(u^n, m^n)$ converges to $(u,m)$. This result illustrates again the robustness of the stable equilibria: if the players deviate from a stable equilibrium configuration,  the learning procedure pushes them back to this equilibrium. 
\bigskip

The techniques used in the paper are inspired by finite dimensional optimal control (see, for instance, in the monograph of Cannarsa and Sinestrari \cite{CannarsaSinestrari}). In particular the fact that the solutions are stable when  restricted  to a subinterval is known in this context. The proof requires a uniqueness result for the solution of the MFG system and its linearized version {\it given an initial condition for $m$ and for $Du$}. To show such a statement we rely on a method developed by Lions and Malgrange for the backward uniqueness of the heat equation, and subsequently extended to systems in Cannarsa and Tessitore \cite{CT}.

\bigskip

The paper is organized in the following way: we first introduce the notation, assumptions and recall standard existence and uniqueness results for the MFG system. We also prove a first new uniqueness result for the MFG system {\it given the initial measure  and the initial vector field}. Then we discuss the notion of potential games (section \ref{sec:potential}) and illustrate the notion by expliciting an example of multiple solutions for a MFG system. In section \ref{sec:stable} we define the notion of stable MFG equilibria  and provide our first main result (Theorem \ref{theo:stable}) on the existence of such equilibria in the potential case. We complete the paper by the analysis of the fictitious play for MFG system (section \ref{sec:fictitious}). In Appendix we prove the uniqueness of a general linear forward-backward system with given initial data: this result is used several times in the text. \\

{\bf Acknowledgement:} The authors were partially supported by the ANR (Agence Nationale de la Recherche) project ANR-16-CE40-0015-01.

\section{Assumptions and basic results on MFG systems}

\subsection{Notation and assumption}\label{subseq:notation}

Throughout the paper we work on the $d-$dimensional torus $\T^d:=\R^d/\Z^d$: this simplifying assumption allows us to ignore issues related to  boundary conditions. We also work in a finite horizon $T>0$. We denote by $\Pk$ the set of Borel probability measures on $\T^d$, endowed with the Monge-Kantororitch distance $\dk$. We define ${\mathcal M}(\T^d,\R^d)$ as the set of Borel vector measures $w$ with finite mass $|w|$. For $\alpha\in [0,1]$, we denote by $C^\alpha([0,T],\Pk)$ the set of maps $m:[0,T]\to\Pk$ which are $\alpha-$Holder continuous if $\alpha\in (0,1)$, continuous if $\alpha=0$, Lipschitz continuous if $\alpha=1$.

Next we recall the notion of derivative of a map $U:\Pk\to \R$ as introduced in \cite{CDLL}. We say that $U$ is $C^1$ if there exists a continuous map $\frac{\delta U}{\delta m}:\T^d\times \Pk\to \R$ such that 
$$
U(m')-U(m)=\int_0^1\inte \frac{\delta U}{\delta m}(x,  (1-t)m+tm')(m'-m)(dx) dt \qquad \forall m,m'\in \Pk.
$$
The derivative is normalized by the condition 
$$
\inte \frac{\delta U}{\delta m}(x, m)dm(x)=0 \qquad \forall m\in \Pk.
$$
We write indifferently $\ds \frac{\delta U}{\delta m}(x,m)(\mu)$ and $\ds \inte \frac{\delta U}{\delta m}(x,m)d\mu(x)$ for a signed measure $\mu$ with finite mass. 

If $u:\T^d\times [0,T]\to \R$ is a sufficiently smooth map, we denote by $Du(x,t)$ and $\Delta u(x,t)$ its spatial gradient and spatial Laplacian and by $\partial_t u(x,t)$ its partial derivative with respect to the time variable. For $p=1, 2, \infty$, we denote by $\|\cdot\|_p$ the $L^p$ norm of a map on $\T^d$. We denote by $C^0$ the set of continuous maps, by $C^{2,0}$ the set of of maps such that $D^2u$ and $\partial_t u$ exist and are continuous. By abuse of notation, we set 
$$
\|u\|_{C^{1,0}}= \|u\|_\infty+ \|D u\|_\infty.
$$
We will also use the classical Holder space. For $\alpha\in (0,1)$, we denote by $C^{0,\alpha}$ the set of map $u=u\in C^0$ which are $\alpha-$Holder continuous in space and $\alpha/2$ in time. The set $C^{1,\alpha}$ is the set of maps $u\in C^0$ such that $u$ and $Du$ belong to $C^{0,\alpha}$. Finally $C^{2,\alpha}$ consists in the maps $u\in C^{2,0}$ such that $D^2u$ and $\partial_tu$ belong to $C^{0,\alpha}$. Let us recall that, if $u$ is in $C^{2,\alpha}$, then $u$ is also in $C^{1,\alpha}$. 
\bigskip

\noindent {\bf Assumptions.} The following assumptions are in force throughout the paper. 

\begin{itemize}
\item The Hamiltonian $H=H(x,p):\T^d\times \R^d\to \R$ is of class $C^2$ and satisfies 
\be\label{cond:Hcoercive}
C^{-1}I_d \leq D^2_{pp}H(x,p)\leq CI_d
\ee
We define the convex conjugate $L$ of $H$ as
\be\label{eq.defL}
L(x,q)= \sup_{p\in \R^d} \{ -p\cdot q-H(x,p)  \}.
\ee

\item The coupling functions $f, g: \T^d\times \Pk\to \R$ are globally Lipschitz continuous with space derivatives $\partial_{x_i} f,\partial_{x_i}g, \partial^2_{x_ix_j}g: \T^d\times \Pk\to \R$ also Lipschitz continuous. In the same way, the measure derivatives $\frac{\delta f}{\delta m}, \frac{\delta g}{\delta m}:\T^d\times \Pk\times \T^d\to \R$ are Lipschitz continuous. 

\end{itemize}

\subsection{The MFG system}

Let us recall that, under our standing assumptions, the MFG system \eqref{eq:MFG} has at least one classical solution: see, for instance, \cite{LL07mf}. In general it is not expected that this solution is unique: uniqueness in known to hold for short time horizon (or small data) or when the coupling functions $f$ and $g$ are monotone \cite{LL07mf}. We provide in the next section an example of multiple solutions. 

However we prove here that there is only one solution given the initial measure {\it and} the initial vector field. This result, which is of limited interest for the true mean field game system, will be used several times in the sequel. 

\begin{Proposition}\label{uniqueMFG} Let $(u_1, m_1)$ and $(u_2,m_2)$ be two solutions of the MFG system \eqref{eq:MFG} with the same initial initial condition for the measure $m_1(t_0)=m_2(t_0)=m_0$ and the same initial vector field  $Du_1(\cdot,t_0)=Du_2(\cdot,t_0)$ in $\T^d$. Then $(u_1,m_1)=(u_2,m_2)$ on $\T^d\times [t_0,T]$. 
\end{Proposition}

\begin{proof} Without loss of generality we assume that $t_0=0$. Let us set, for $i=1, \dots, d$,  $w_i= \partial_{x_i}(u_2-u_1)$, $w= (w_i)_{i=1, \dots, d}$, $\mu= m_2-m_1$. Then $(w,\mu)$ solves the system
$$
\left\{\begin{array}{l}
-\partial_t w_i -\Delta w_i +g_i=0 \; {\rm in}\; \T^d\times (0,T), \; i=1, \dots, d,\\
\partial_t \mu -\Delta \mu + \dive(h)= 0 \; {\rm in}\; \T^d\times (0,T) \\
w_i(x,0)=\mu(x,0)=0 \; {\rm in}\; \T^d, \; i=1, \dots, d.
\end{array}\right.
$$
where 
$$
\begin{array}{rl}
\ds g_i(x,t) \; =  & \ds  \partial_{x_i}H(x,Du_2)+ D_pH(x, Du_2)\cdot D(\partial_{x_i}u_2) \\
 & \ds - \partial_{x_i}H(x,Du_1)- D_pH(x, Du_1)\cdot D(\partial_{x_i}u_1)
+ \partial_{x_i} f(x,m_2)- \partial_{x_i}f(x,m_1)
\end{array}
$$
and
$$
h(x,t)= m_2 D_pH(x, Du_2(x,t))-m_1 D_pH(x, Du_1(x,t)) \: . 
$$
As the $u_i$ are classical solutions and, by our assumption, $\partial_{x_i}f: \T^d\times \Pk\to \R$ is globally Lipschitz continuous,  we have
$$
\begin{array}{rl}
\ds \left|g_i(x,t)\right| \; \leq  & \ds  C\left( |D(u_2-u_1)(x,t)|+ |D(\partial_{x_i}(u_2-u_1))(x,t)|  + \|m_2(\cdot,t)-m_1(\cdot,t)\|_2 \right) \\
\leq & \ds  C\left( |w(x,t)|+ |Dw(x,t)|  + \|\mu(\cdot,t)\|_2 \right).
\end{array}
$$
In the same way, 
$$
\begin{array}{rl}
\ds \left| h(x,t)\right| \leq  & \ds C\left(  |(m_2-m_1)(x,t)|+ |D(u_2-u_1)(x,t)| \right) \\
\leq & \ds C \left( |w(x,t)|+ |\mu(x,t)| \right).
\end{array}
$$
Moreover, as $\|Du_i\|_\infty+ \|D^2u_i\|_\infty+ \|m_i\|_\infty +\|Dm_i\|_\infty \leq C$, we also obtain  
$$
\begin{array}{rl}
\ds \left| \dive (h)(x,t)\right| \leq  & \ds C \left( |w(x,t)|+|Dw(x,t)|+ |\mu(x,t)|+ |D\mu(x,t)| \right). \\
\end{array}
$$
Then we can conclude from Theorem \ref{theo:continuation} in the Appendix that $\mu=0$ on $\T^d\times [0,T]$. Since $m_1=m_2$ and since the backward Hamilton-Jacobi equations for the $u_i$ have the same right-hand side, their solution is the same and we conclude that $u_1=u_2$. 
\end{proof}

\section{Potential MFG systems}\label{sec:potential}

Throughout this section we assume that the MFG system is {\it potential}, i.e., the coupling functions $f$ and $g$ derive from potentials: there exists $F,G:\Pk\to\R$ of class $C^1$ such that 
\be\label{eq:potential}
f=\frac{\delta F}{\delta m}, \qquad g=\frac{\delta G}{\delta m}.
\ee
The aim of the section is to recall that, under this condition, the MFG systems is then potential (i.e., the equilibria are obtained through a minimization procedure). We derive from this an example of non uniqueness of the solution of the MFG system.

\subsection{Solutions of the MFG system and minimizers of the potential}

Let  $t_0\in [0,T]$ and $m_0\in \Pk$. Our aim is to define the functional $J$
$$
J(m,w)=\int_{t_0}^T\inte L\left(x, \frac{w(x,t)}{m(x,t)}\right)m(x,t) dxdt+ \int_{t_0}^T F(m(t))dt + G(m(T))
$$
under the constraint
 $$
 \partial_t m-\Delta m +\dive(w)= 0 \; {\rm in }\; \T^d\times [t_0,T], \:  \: m(t_0)=m_0,
  $$
for a large class of data $(m,w)$. This classical construction is reminiscent of the Benamou and Brenier approach of optimal transport and can be found, for instance, in \cite{AGS}. 

We denote by ${\mathcal E}(t_0)$ the  set of  pairs time dependent Borel measures $(m(t),w(t))\in \Pk\times {\mathcal M}(\T^d,\R^d)$ such that $t\to m(t)$ is continuous, 
$$
\int_0^T |w(t)|dt <\infty, 
$$
and equation
$$
 \partial_t m-\Delta m +\dive(w)= 0 \; {\rm in }\; \T^d\times [t_0,T], \:  \: m(t_0)=m_0
 $$
holds in the sense of distribution. We also denote by ${\mathcal E}_2(t_0)$ the subset of $(m(t),w(t))\in {\mathcal E}(t_0)$ such that $w(t)$ is absolutely continuous with respect to $m(t)$ with a density $\frac{dw(t)}{dm(t)}$ satisfying 
$$
\inte  \int_0^T \left| \frac{dw(t)}{dm(t)}(x)\right|^2 m(dx,t)dt<\infty.
$$
Then we defined $J$ on ${\mathcal E}(t_0)$ by 
$$
J(m,w):=
\left\{
\begin{array}{l}
\displaystyle  \int_{t_0}^T\inte L\left(x, \frac{dw(t)}{dm(t)}(x)\right)m(dx,t) dt+ \int_{t_0}^T F(m(t))dt + G(m(T)) \\
\qquad \qquad  \qquad \qquad \qquad \qquad \qquad \qquad \qquad \qquad \qquad \qquad   \mbox{ if  $(m,w)\in {\mathcal E}_2(t_0)$}   \\
  \\
   +\infty \qquad \qquad \qquad \qquad \qquad \qquad  \mbox{ otherwise} 
\end{array}
\right.
$$
When it  will be important to stress the dependence on the initial data we will write  $J(t_0, m_0, \cdot,\cdot)$.

The following result states that minimizers of the functional $J$ correspond to solution of the MFG system. This remark was first pointed out in \cite{LL07mf} and used repetitively since then in different contexts. 

\begin{Proposition}\label{prop:minMFG} Under our standing assumptions:  
\begin{itemize}
\item[(i)] For any $t_0\in [0,T]$ and $m_0\in \Pk$ there exists a minimum $(m,w)  \in {\mathcal E}_2$  of  $J(t_0,m_0, \cdot,\cdot)$.
\item[(ii)] Let $(m,w)$ be minimum of $J(t_0,m_0, \cdot,\cdot)$. Then there exists $u$ such that $(u,m)$ is a classical solution to the MFG system 
\be\label{MFGprop}
\left\{\begin{array}{l}
-\partial_t u -\Delta u +H(x,Du)=f(x,m) \; {\rm in}\; \T^d\times (t_0,T), \\
\partial_t m -\Delta m - \dive(mD_pH(x, Du))= 0 \; {\rm in}\; \T^d\times (t_0,T) \\
m(x,t_0)=m_0(x), \; u(x,T)=g(x,m(T))  \; {\rm in}\; \T^d 
\end{array}\right.
\ee
and $ w(x,t)= -m(x,t) D_pH(x,Du)$. 
In particular, any minimizer is a classical solution of the above system, i.e., $u,m\in C^{2,\alpha}$.  
\end{itemize}
\end{Proposition}

The proof is standard and is closely related to techniques used in optimal transport theory. However, as it has never been explicitly checked in this specific framework where the potentials are non convex, we provide the main argument for the sake of completeness. 

\begin{proof} 
To prove (i) let us consider a minimizing sequence  $(m_n,w_n)\in {\mathcal E}_2$. By construction 
$J(m_n,w_n) \leq C$, thus the coercivity assumption \eqref{cond:Hcoercive} on $H$ implies the following uniform bound: 
\begin{equation} \label{eqbound}
  \int_{t_0}^T\inte \left| \frac{dw_n(t)}{d m_n(t)}(x)\right|^2 m_n(dx,t)dt \leq C \: . 
\end{equation}
We can then argue as in \cite[Lemma 3.1]{CGPT} to conclude that the sequence $(m_n)$ is uniformly bounded in $C^{1/2}([0,T],\Pk)$. In particular,   
$$
  \int_{t_0}^T |w_n(t)| dt \leq 
     \Big(   \int_{t_0}^T\inte \left|\frac{d  w_n(t)}{d m_n(t)}\right|^2m_n(dx,t)dt  \Big)^{1/2}  
  \Big(   \int_{t_0}^T\inte m_n(dx,t) dt     \Big)^{1/2}  \leq C.
  $$
So, up to a subsequence, $w_n \rightarrow w$ in ${\mathcal M}((0,T) \times \T^d, \R^d)$ and $(m_n)$ converges in $C^0([0,T], \Pk)$.  By standard argument the pair $(m,w)$ belongs to ${\mathcal E}_2$.
To conclude that the couple $(m,w)$ is a minimum of the functional $J$  we remark that the functional $J$  is 
lower semicontinuous on ${\mathcal E}$.

In order to prove (ii) we define on ${\mathcal E}$
$$
 \Phi(m,w):=
 \left\{
\begin{array}{ll}
\displaystyle  \int_{t_0}^T\inte L\left(x, \frac{ w(t)}{ m(t)}\right)m(dx,t)dt 
  &  \mbox{ if  $(m,w)\in {\mathcal E}_2$}   \\
  \\
   +\infty & \mbox{ otherwise} 
\end{array}
\right.
$$ 
and 
$$
 \Psi(m):= \int_{t_0}^T F(m(t))dt + G(m(T)) .
$$
Note that $J(m,w)= \Phi(m,w)+ \Psi(m)$. Let $(\bar{m},\bar{w})$ be a minimum of $J$. Recall that $\bar m\in C^{1/2}([0,T], \Pk)$. We first claim  that,  for any 
$(m,w) \in {\mathcal E}$, we have 
\begin{equation} \label{eqlemme2} 
\begin{array}{r}
\ds -\int_{t_0}^T\inte f(x,\bar{m}(t))(m(dx,t)-\bar{m}(dx,t)) - \inte g(x,\bar{m}(T))(m(dx,T)-\bar{m}(dx,T)) \qquad\\
\qquad \qquad \qquad \qquad \ds  \leq \Phi(m,w)-\Phi(\bar{m},\bar{w})  \: . 
\end{array}
\end{equation}
Indeed, setting $m_\lambda:=(1-\lambda)\bar{m}+\lambda m$, $ w_\lambda:=(1-\lambda)\bar{w}+\lambda w$ , $\lambda \in (0,1)$ we have by minimality of $(\bar m,\bar w)$:
$$
\Phi(m_\lambda,w_\lambda)-\Phi(\bar{m},\bar{w}) \geq \Psi(\bar{m})-\Psi(m_\lambda).
$$ 
Thus, by the regularity assumptions on $F$ and $G$ and the convexity of $\Phi$, we obtain 
$$
\begin{array}{l}
\ds  \lambda(\Phi(m,w)-\Phi(\bar{m},\bar{w})  ) \\
\; \ds \geq \lambda \Big( -\int_{t_0}^T\inte f(x,\bar{m}(t,x))(m-\bar{m})(dx,t) - \inte g(x,\bar{m}(x,T))(m-\bar{m})(dx,T))  \Big)+o(\lambda).
\end{array}
$$
Then \eqref{eqlemme2} follows dividing by $\lambda$ and letting $\lambda$ tends to $0$. 

By \eqref{eqlemme2}, the pair $(\bar m,\bar w)$ is a minimizer of the following (local and convex) functional on ${\mathcal E}$: 
$$
\tilde J(m,w)= \Phi(m,w)+ \int_0^T\inte f(x,\bar{m}(t))m(dx,t)dt + \inte g(x,\bar{m}(T))m(dx,T).
$$
We can then follow the standard arguments for convex functionals (see for instance \cite{CGPT}): the problem of minimizing $\tilde J$ on ${\mathcal E}$ is the dual problem (in the sense of the  Fenchel-Rockafellar duality theorem) of the problem
$$
\inf_{u \in  C^2} \Big\{   -\inte m_0(x)u(x,0)dx \: : \:   -\partial_t u -\Delta u +H(x,Du) \leq f(x,\bar{m}(t))  \: \mbox{ and }\: u(x,T)\leq g(x,\bar{m}(x,T))   \Big\}. 
$$
By comparison, there is an obvious  minimum to this problem which is the  solution $\bar u$ to 
$$
\left\{\begin{array}{l}
-\partial_t \bar{u} -\Delta \bar{u} +H(x,D\bar{u}) = f(x,\bar{m}(t))\; {\rm in}\; \T^d\times (0,T)\\
 u(x,T)= g(x,\bar{m}(x,T))\; {\rm in}\; \T^d \: . 
 \end{array}\right.
 $$
 This solution is $C^{2,\alpha}$ because $\bar m\in C^{1/2}([0,T), \Pk)$. 
By the Fenchel-Rockafellar duality theorem,  see also   \cite[Lemma 3.2]{CGPT},  we have that  
$$
0=  \tilde J(\bar m, \bar w)-\inte m_0(x)\bar{u}(x,0)dx \:.
$$
This means that 
$$
\begin{array}{l}
\ds 0 = \int_{t_0}^T\inte \left(L\left(x, \frac{d w(t)}{d m(t)}\right) +f(x,\bar m(t)) \right) \bar m(dx,t)dt +\inte g(x,\bar m(T))\bar m(dx,T)-\inte m_0(x)\bar{u}(x,0) dx \\
\qquad \ds  = \int_{t_0}^T\inte \left(L\left(x, \frac{d w(t)}{d m(t)}\right)+ (-\partial_t \bar u-\Delta \bar u +H(x,D\bar u)) \right) \bar m(dx,t)dt +\inte g(x,\bar m(T))\bar m(dx,T)\\
\qquad \qquad \qquad \qquad \qquad \qquad \qquad \ds -\inte m_0(x)\bar{u}(x,0) dx \\
\qquad \ds = \int_{t_0}^T\inte \left(L\left(x, \frac{d w(t)}{d m(t)}\right)+ H(x,D\bar u)+Du\cdot (\frac{d\bar w(t)}{d\bar m(t)}) \right) \bar m(dx,t)dt
\end{array}
$$
where we used the equation for $\bar u$ in the second equality and the equation for $(\bar m,\bar w)$ in the last one. Recalling that $L$ is the convex conjugate of $H$ which is uniformly convex, we find
$$
 \frac{d\bar w(t)}{d\bar m(t)} = -D_pH(x, D\bar u) \qquad \bar m-{\rm a.e.}
 $$
This means that $\bar m$ solves the Kolmogorov equation 
$$
\partial_t \bar m-\Delta \bar m -\dive(\bar m  D_pH(x, D\bar u))=0, \qquad \bar m(\cdot, t_0)=m_0,
$$
which has a regular drift: thus $\bar m$ is  of class $C^{2,\alpha}$ by Schauder theory. Therefore $\bar w$ is also smooth and the proof of (ii) is complete. 
 \end{proof}

We now explain that, if $(m,w)$ is a  minimizers of $J(t_0, m_0,\cdot,\cdot)$, then, for any later time $t_1>t_0$, $(m,w)$ is the unique minimizer for $J(t_1,m(t_1), \cdot,\cdot)$. The argument borrows ideas in finite dimensional control theory \cite{CannarsaSinestrari}. 

\begin{Proposition} Let $(m,w)$ be  minimum of $J(t_0,m_0,\cdot,\cdot)$ for a given initial condition $m_0$ at time $t_0$. Then, for any $t_1\in (t_0,T)$, $J(t_1, m(t_1), \cdot,\cdot)$ has a unique minimum, which is the restriction to $[t_1,T]$ of $(m,w)$.  
\end{Proposition}

\begin{proof} By dynamic programming principle, the restriction to $[t_1,T]$ of $(m,w)$ is a minimum of $J(t_1, m(t_1), \cdot,\cdot)$.
Let $(\tilde m, \tilde w)$ be another minimum of $J(t_1, m(t_1), \cdot,\cdot)$. Then the following map
$$
(\hat m, \hat w)=\left\{\begin{array}{ll}
(m,w) & {\rm on }\; \T^d\times [t_0,t_1)  \\
(\tilde m, \tilde w) & {\rm on }\; \T^d\times [t_1,T]
\end{array}\right.
$$
is also a minimum of $J(t_0,m_0,\cdot,\cdot)$. Thus, by Proposition \ref{prop:minMFG}, there exist $u$, $\hat u$ such that $(u,m)$ and $(\hat u,\hat m)$ are both classical solutions to the MFG system \eqref{MFGprop}, with $w= -mD_pH(\cdot, Du)$ and $\hat w= -\hat mD_pH(\cdot, D\hat u)$. On $\T^d\times (t_0,t_1)$, we have 
$m=\hat m>0$ and $-\hat mD_pH(\cdot, D\hat u)= -mD_pH(\cdot, Du)$, so that, by uniform convexity of  $H$, we obtain the equality $Du=D\hat u$ on $\T^d\times (t_0,t_1)$. By continuity we infer that $Du(\cdot,t_1)= D\hat u(\cdot, t_1)$. Then  Proposition \ref{uniqueMFG} implies that  $(\hat u, \hat m)= (\tilde u, \tilde m)=(u,m)$ on $\T^d\times [t_1,T]$. 
\end{proof}

\subsection{An example of multiple solutions}

As an application of Proposition \ref{prop:minMFG}, we provide here a simple example of multiple solutions for the MFG system. Such examples in the time dependent setting are scarce: we are only aware of a recent result by Bardi and Fischer~\cite{BaFi}. 
The idea is relatively elementary: we suppose that the MFG system is symmetric with respect to the $x_1$ variable, but that the potential favors asymmetric solutions. Then the MFG system enjoys a symmetric solution as well as an asymmetric one. 

To build the example, let $\tau(x_1,x_2,\dots, x_d)=(-x_1, x_2, \dots, x_d)$ (viewed as a map on $\R^d$ as well as on $\T^d$) and let us denote by $\Pks$ the  subset of measures $m\in \Pk$ such that $\tau\sharp m=m$. We suppose, in addition to our general conditions on $H$, $f$ and $g$,  that:
\begin{itemize}
\item[(1)] The function $F$ is nonnegative and symmetric with respect to the $x_1$ variable: $F(\tau\sharp \: m)=F(m)$, $\forall m\in \Pk$.  Note that this implies that $f(\tau(x), m)=f(x,m)$ for any $m\in \Pks$. 
\item[(2)] There exists a measure $\mu\in\Pk$ such that 
$\ds \inf_{m\in \Pks} F(m)> F(\mu).$
\item[(3)] The Hamiltonian $H$ is symmetric with respect to the $x_1$ variable: $H(\tau(x),\tau(p))=H(x,p)$ for any $(x,p)\in \T^d\times \R^d$. 
Moreover $L$ is nonnegative. 
\end{itemize}
Given a parameter $\theta>0$, we consider the MFG system 
\be\label{MFGpropUNICITA}
\left\{\begin{array}{l}
-\partial_t u -\Delta u +H(x,Du)=\theta \: f(x,m) \; {\rm in}\; \T^d\times (0,T), \\
\partial_t m -\Delta m - \dive(mD_pH(x, Du))= 0 \; {\rm in}\; \T^d\times (0,T) \\
m(x,0)=m_0(x), \; u(x,T)=0  \; {\rm in}\; \T^d.
\end{array}\right.
\ee
 
\begin{Proposition}  \label{cosNONuni}  Under our assumptions, for any symmetric initial condition $m_0\in \Pks$  with a smooth density, there exist $T>0$ and $\theta>0$ large enough such that  the MFG system \eqref{MFGpropUNICITA} has at least two different solutions. 
\end{Proposition}

\begin{proof} Let us first explain that the MFG system has a symmetric solution. For this we build, in a standard way, a fixed point mapping $\Phi$ on $C^\alpha([0,T], \Pks)$ (for some fixed $\alpha\in (0,1/2))$: see for instance \cite{LL07mf}. To $m\in  C^\alpha([0,T], \Pks)$ we first associate the solution to 
$$
\left\{\begin{array}{l}
-\partial_t u -\Delta u +H(x,Du)=\theta \: f(x,m) \; {\rm in}\; \T^d\times (0,T), \\
u(x,T)=0  \; {\rm in}\; \T^d.
\end{array}\right.
$$
Because of our assumptions, the above equation has a unique classical solution $u$, which is therefore symmetric with respect to $x_1$. Moreover, thanks to the uniform convexity of $H$,  $Du$ is bounded independently of $m$. Next we consider the solution $\tilde m$ to the Kolmogorov equation
$$
\left\{\begin{array}{l}
\partial_t \tilde m -\Delta \tilde m - \dive(\tilde mD_pH(x, Du))= 0 \; {\rm in}\; \T^d\times (0,T) \\
\tilde m(x,0)=m_0(x)\; {\rm in}\; \T^d.
\end{array}\right.
$$
Here again there exists a classical solution $\tilde m$, which is symmetric with respect to $x_1$. Moreover, as $Du$ is bounded, $\tilde m$ is bounded in $C^{1/2}([0,T], \Pks)$ independently of $m$. Finally we set $\Phi(m):=\tilde m$. It is easy to check that $\Phi$ has a fixed point, which leads to a symmetric solution to the MFG system. We denote by $(\hat u,\hat m)$ this symmetric solution and set $\hat w:= -\hat m D_pH(x,D\hat u)$. 

Next we show that the symmetric solution cannot minimize the functional $J$ for suitable choice of $T$ and $\theta$. 
Let us denote by $M$  the minimum value of the functional $F$ over $\Pks$. Recalling that there exists $\mu$ such that $F(\mu)<M$. We fix  $\varepsilon, \eta >0$ such that 
$$ 
M > \varepsilon + \max_{ \dk(m,\mu)\leq \eta} F(m).
$$ 
We then  construct (e.g. by convolution) a  measure  $\bar{m}\in \Pk$ which is absolutely continuous and such that  (for some constant $C>0$)
$$
 \bar{m} \in C^{\infty}(\T^d) \: , \: \: \dk(\bar{m},\mu)\leq \eta  \: , \: \: \frac{1}{C} \leq  \bar{m}(x) \leq C. 
$$
Setting $\bar{w}=D \bar{m}$, we have $-\Delta \bar{m} - \dive(\bar{w})= 0$.  Then we choose $\theta\geq 1$ large enough to have 
$$
\frac{\varepsilon \theta}{2} > \inte L\left(x, \frac{\bar{w}(x)}{\bar{m}(x)}\right) \bar{m}(x)dx.
$$ 
Next we define the pair $(m,w)$ which connects $m_0$ and $\bar m$ in time $1$ and is equal to $(\bar m, \bar w)$ after time $1$: 
$$
(m(x,t),w(x,t)):=\left((1-t)m_0(x)+t\bar m(x), -D\phi(x)+(1-t)D m_0(x)+ tD \bar m(x)\right) \; {\rm for}\; t\in [0,1], 
$$
(where $\phi$ is a solution to $-\Delta \phi= m_0-\bar m$ in $\T^d$) and
$$
(m(x,t), w(x,t))= (\bar m(x),\bar w(x))  \qquad {\rm for}\; t\in [1,T].
$$
Then one easily checks that the pair $(m,w)$ belongs to ${\mathcal E}_2$. Finally we choose $T$ large enough to have
$$
\sup_{m'\in \Pk} F(m')+ \int_0^1\inte m(x,t)L\Big(x,\frac{w(x,t)}{m(x,t)}\Big)dxdt \leq \frac{ \ep \theta T}{2}.
$$
As $L\geq 0$, we have  
$$
 J(\hat{m},\hat{w}) \geq \theta \: T \: M  > \theta T \varepsilon + \theta T \max_{ d_1(m',\mu)\leq \eta} F(m').
 $$
On the other hand, 
$$
\begin{array}{rl}
\ds J(m,w)\; = & \ds \int_0^1  \left(\inte  L(x,\frac{w}{m}) m dx  +F(m(t)) \right) dt + (T-1) \left(\inte  L(x, \frac{\bar{w}}{\bar{m}}) \bar{m}dx+ F(\bar m)\right)\\
< & \ds \ep \theta T/2+(T-1) \ep\theta/2 +(T-1) \max_{ d_1(m',\mu)\leq \eta} F(m') \; <\;  J(\hat{m},\hat{w}).
\end{array}
$$
This proves that $(\hat m,\hat w)$ is not a minimizer of $J$. As, by Proposition \ref{prop:minMFG}, $J$ has a minimum which is associated with a solution of the MFG system \eqref{MFGpropUNICITA}, there exists a solution to the MFG system \eqref{MFGpropUNICITA} different from $(\hat u,\hat m)$.
\end{proof}

\section{Stable solution to the MFG system}\label{sec:stable}

\subsection{Definition and basic property}

We say that a solution of the MFG system \eqref{eq:MFG} is stable if the unique solution to linearized system is the trivial one. 

\begin{Definition} Let $(u,m)$ be a solution to the MFG system \eqref{eq:MFG} with initial condition $(t_0,m_0)\in [0,T]\times \Pk$. We say that the solution $(u,m)$ is stable if $(v,\mu)=(0,0)$ is the unique solution to the linearized system 
$$
\left\{\begin{array}{l}
-\partial_t v -\Delta v +D_pH(x,Du)\cdot Dv=\frac{\delta f}{\delta m}(x,m)(\mu) \; {\rm in}\; \T^d\times (t_0,T),\\
\partial_t \mu -\Delta \mu - \dive(\mu D_pH(x, Du))-\dive (mD^2_{pp}H(x,Du)Dv)= 0 \; {\rm in}\; \T^d\times (t_0,T) \\
\mu(x,t_0)=0, \; v(x,T)=\frac{\delta g}{\delta m}(x,m(T))(\mu(T))  \; {\rm in}\; \T^d.
\end{array}\right.
$$
\end{Definition}

Before proving that stable solutions do exist, we show that they are isolated. 

\begin{Proposition}\label{prop:stable} Let $(u,m)$ be a stable solution starting from an initial position $(t_0,m_0)\in [0,T]\times \Pk$. Then there is $\eta>0$ such that, for any $m_0^1\in \Pk$ with $\|m_0^1-m(t_0)\|_{C^0}\leq \eta$, there exists at most one solution $(u_1,m_1)$ of the MFG system 
$$
\left\{\begin{array}{l}
-\partial_t u_1 -\Delta u_1 +H(x,Du_1)=f(x,m_1) \; {\rm in}\; \T^d\times (t_0,T), \\
\partial_t m_1 -\Delta m_1 - \dive(m_1D_pH(x, Du_1))= 0 \; {\rm in}\; \T^d\times (t_0,T) \\
m_1(x,t_0)=m_0^1(x), \; u_1(x,T)=g(x,m_1(T))  \; {\rm in}\; \T^d.
\end{array}\right.
$$
 such that 
$$
\|(u,m)-(u_1,m_1)\|_{C^{1,0}\times C^0} \leq \eta.
$$ 
\end{Proposition}

\begin{proof} We argue by contradiction, assuming that there exists two distinct solutions $(u^{n,1},m^{n,1})$ and  $(u^{n,2},m^{n,2})$ of the MFG system with $m^{n,1}(\cdot,t_0)= m^{n,2}(\cdot,t_0)$ and converging to $(u,m)$ as $n\to +\infty$ in $C^{1,0}\times C^0$. 

We set 
$$
(v^n, \mu^n)= (\rho^n)^{-1}\left((u^{n,2},m^{n,2})- (u^{n,1},m^{n,1})\right), 
$$
where
$$
\rho^n:= \|(u^{n,2},m^{n,2})- (u^{n,1},m^{n,1})\|_{C^{1,0}\times C^0}.
$$
We note that the pair $(v^n, \mu^n)$ solves
$$
\left\{\begin{array}{l}
\ds -\partial w-\Delta w +g^n=0  \; {\rm in}\; \T^d\times (t_0,T),\\
\ds \partial \mu-\Delta \mu +\dive(h^n)=0 \; {\rm in}\; \T^d\times (t_0,T),
\end{array}\right.
$$
where 
$$
g^n(x,t)= (\rho^n)^{-1} \left( H(x,Du^{n,2}(x,t))-H(x,Du^{n,1}(x,t)) -f(x,m^{n,2}(t))+f(x,m^{n,1}(t))\right)
$$
and 
$$
h^n(x,t)= (\rho^n)^{-1}\left(- m^{2,n}(x,t)D_pH(x, Du^{n,2}(x,t))+m^{1,n}(x,t)D_pH(x, Du^{n,1}(x,t))\right).
$$
Note that $\|g^n\|_\infty$ and $\|h^n\|_\infty$ are bounded by definition of $\rho^n$. Thus, by standard arguments in uniform parabolic equations
(see for instance, Theorem 9.1 and Theorem 10.1 in \cite{LSU}), there exists $\alpha\in (0,1)$ such that $(v^n)$ is bounded in $C^{1+\alpha, (1+\alpha)/2}$ while $(\mu^n)$ is bounded in $C^{\alpha,\alpha/\varepsilon}$. Thus, up to a subsequence denoted in the same way, $(v^n,\mu^n)$ converges in $C^{1,0}\times C^0$ to some $(v,\mu)$. Note that
$$
\|(v,\mu)\|_{C^{1,0}\times C^0}=1
$$
and that  $(v,\mu)$ is a (a priori weak) solution to the linearized problem
$$
\left\{\begin{array}{l}
-\partial_t v -\Delta v +D_pH(x,Du)\cdot Dv=\frac{\delta f}{\delta m}(x,m)(\mu) \; {\rm in}\; \T^d\times (t_0,T), \\
\partial_t \mu -\Delta \mu - \dive(\mu D_pH(x, Du))-\dive (mD^2_{pp}H(x,Du)Dv)= 0 \; {\rm in}\; \T^d\times (t_0,T) \\
\mu(x,t_1)=0, \; v(x,T)=\frac{\delta g}{\delta m}(x,m(T))(\mu(T))  \; {\rm in}\; \T^d.
\end{array}\right.
$$
By parabolic regularity, $(v,\mu)$ is actually a classical solution to the above equation. As $(u,m)$ is a stable solution, one must have $(v,\mu)=(0,0)$. This leads to a contradiction with the fact that $\|(v,\mu)\|_{C^{1,0}\times C^0}=1$. 
\end{proof}

\subsection{Existence of stable solutions for potential MFG systems}

The next result states that, if $(u,m)$ is a solution to a potential MFG system on the time interval $[t_0,T]$ which corresponds to a minimizer of $J$, then the restriction of $(u,m)$ to any subinterval $[t_1,T]$ (where $t_1\in (t_0,T)$) is a stable solution of the MFG system. 

\begin{Theorem}\label{theo:stable} Let us assume that the game is potential, i.e, satisfies \eqref{eq:potential} for some $C^1$ maps $F,G:\Pk\to\R$. Let $(m,w)$ be  minimum of $J(t_0,m_0,\cdot,\cdot)$ for a given initial condition $m_0\in\Pk$ at time $t_0\in [0,T]$. Let $u$ be such that $(u,m)$ is a solution to the MFG system \eqref{eq:MFG}. Then, for any $t_1\in (t_0,T)$, the restriction of $(u,m)$ to the time interval $[t_1,T]$ is a stable solution to the MFG system. 
\end{Theorem}

Note that the Theorem implies that there are many stable solutions to the MFG system. In particular, the set of initial conditions $(t_0,m_0)$ for which there is a stable solution is dense. 
\bigskip

We now start the proof of Theorem \ref{theo:stable}.
We have to show that,  for any $t_1\in (t_0,T)$, $(v,\mu):=(0,0)$ is the unique classical solution to the linearized system 
\be\label{eq.LS}
\left\{\begin{array}{l}
\ds -\partial_t v -\Delta v +D_pH(x,Du)\cdot Dv=\frac{\delta f}{\delta m}(x,m)(\mu) \; {\rm in}\; \T^d\times (t_1,T), \\
\ds \partial_t \mu -\Delta \mu - \dive(\mu D_pH(x, Du))-\dive (mD^2_{pp}H(x,Du)Dv)= 0 \; {\rm in}\; \T^d\times (t_1,T) \\
\ds \mu(x,t_1)=0, \; v(x,T)=\frac{\delta g}{\delta m}(x,m(T))(\mu(T))  \; {\rm in}\; \T^d.
\end{array}\right.
\ee
The proof requires several steps. 

We begin by computing the directional derivative of the map $J$. For this, let us fix an admissible direction $(\mu, z)$. We assume for a while that $(\mu, z)$ is smooth and satisfies
\be\label{eq.cont}
\partial_t \mu-\Delta \mu +\dive(z)= 0 \; {\rm in }\; \T^d\times (t_0,T), \qquad \mu(t_0)=0.
\ee
We also assume that $\mu=0$ and $z=0$ on the time interval $[t_0, t_0+\ep]$ for some $\ep>0$. 
We then consider the map $h\to J((m,w)+h(\mu,z))$ for $|h|$ small. As $m>0$ on $(t_0, T]$ and $\mu=0$ on $[t_0,t_0+\ep]$, we have $m+h\mu>0$ for $|h|$ small enough. 
Then, by optimality of $(m,w)$ for $J$, we have: 
$$
\frac{d}{dh}_{|_{h=0}} J((m,w)+h(\mu,z))=0, \qquad \frac{d^2}{dh^2}_{|_{h=0}} J((m,w)+h(\mu,z))\geq 0, 
$$
where 
$$
\begin{array}{rl}
\ds \frac{d}{dh}_{|_{h=0}} J((m,w)+h(\mu,z)) \; = & \ds  
\int_{t_0}^T\inte  \mu L(x, \frac{w}{m})+ m D_qL\left(x, \frac{w}{m}\right)\cdot (\frac{z}{m}-\frac{\mu w}{m^2}) 
\\
& \ds + \int_{t_0}^T\inte  f(x,m(t))\mu 
+
\inte g(x,m(T))\mu(T)
\end{array}
$$
and 
$$
\begin{array}{l}
\ds \frac{d^2}{dh^2}_{|_{h=0}} J((m,w)+h(\mu,z)) \\ 
\qquad = \ds 
\int_{t_0}^T\inte  2\mu D_qL(x, \frac{w}{m})\cdot (\frac{z}{m}-\frac{\mu w}{m^2}) +m D_qL\left(x, \frac{w}{m}\right)\cdot (-\frac{\mu z}{m^2}+2\frac{\mu^2 w}{m^3}-\frac{\mu z}{m^2})
 \\
\qquad \qquad  \ds 
+ \int_{t_0}^T\inte  m D^2_{qq}L\left(x, \frac{w}{m}\right) (\frac{z}{m}-\frac{\mu w}{m^2}) \cdot (\frac{z}{m}-\frac{\mu w}{m^2}) \\
\qquad \qquad  \ds + \int_{t_0}^T\inte \inte \frac{\delta f}{\delta m} (x,m(t),y)\mu(x,t)\mu(y,t) +
\inte \inte \frac{\delta g}{\delta m} (x,m(T),y)\mu(x,T)\mu(y,T).
\end{array}
$$
Recalling that $w= -mD_pH(x,Du)$, we can rearrange the expressions: 
$$
\begin{array}{rl}
\ds \frac{d}{dh}_{|_{h=0}} J((m,w)+h(\mu,z)) \; = & \ds  
\int_{t_0}^T\inte  \mu L(x, \frac{w}{m})+ D_qL\left(x, \frac{w}{m}\right)\cdot (z+ \mu D_pH(x,Du))
 \\
& \ds +\int_{t_0}^T\inte  f(x,m(t))\mu
+
\inte g(x,m(T))\mu(T)
\end{array}
$$
and 
$$
\begin{array}{l}
\ds \frac{d^2}{dh^2}_{|_{h=0}} J((m,w)+h(\mu,z)) \\ 
= \ds 
\int_{t_0}^T\inte   m^{-1} D^2_{qq}L\left(x, \frac{w}{m}\right) (z+  \mu D_pH(x,Du)) \cdot (z+  \mu D_pH(x,Du)) \\
\qquad \qquad  \ds + \int_{t_0}^T\inte \inte \frac{\delta f}{\delta m} (x,m(t),y)\mu(x,t)\mu(y,t) +
\inte \inte \frac{\delta g}{\delta m} (x,m(T),y)\mu(x,T)\mu(y,T).
\end{array}
$$
We denote by ${\mathcal J}(t_0,m,u; \mu,z)$ the right-hand side of the above expression: 
$$
\begin{array}{l}
{\mathcal J}(t_0,m,u; \mu,z)\\
\qquad \ds := \int_{t_0}^T\inte   m^{-1} D^2_{qq}L\left(x, \frac{w}{m}\right) (z+  \mu D_pH(x,Du)) \cdot (z+  \mu D_pH(x,Du)) \\
\qquad \qquad  \ds + \int_{t_0}^T\inte \inte \frac{\delta f}{\delta m} (x,m(t),y)\mu(x,t)\mu(y,t) +
\inte \inte \frac{\delta g}{\delta m} (x,m(T),y)\mu(x,T)\mu(y,T).
\end{array}
$$
Note that ${\mathcal J}(t_0,m,u; \mu,z)$ is defined for $\mu,z\in L^2(\T^d\times (0,T))$. 
By regularization, one has therefore: 
$$
{\mathcal J}(t_0,m,u;\mu,z) \geq 0 
$$
for any $\mu,z\in L^2(\T^d\times (t_0,T))$ such that \eqref{eq.cont} holds in the sense of distribution. 
Let us also recall \cite{CDLL} that the map $(x,y)\to \frac{\delta f}{\delta m}(x,m,y)$ is not symmetric, but satisfies the relation: 
$$
\frac{\delta f}{\delta m}(x,m,y)= \frac{\delta f}{\delta m}(y,m,x)+f(x,m)-f(y,m). 
$$
So, for any $\mu,\mu'$ such that $\ds \inte \mu=\inte \mu'=0$, we have 
\be\label{rel.sym}
\inte\inte  \frac{\delta f}{\delta m}(x,m,y)\mu(x)\mu'(y)= \inte\inte \frac{\delta f}{\delta m}(x,m,y)\mu(y)\mu(x). 
\ee

\begin{Lemma} \label{lem:equiv} Let $t_1\in [t_0, T)$. For any $(\mu,z)\in L^2(\T^d\times (t_1,T))$ such that 
\be\label{eq.cont.t1}
\partial_t \mu-\Delta \mu +\dive(z)= 0 \; {\rm in }\; \T^d\times (t_1,T), \qquad \mu(t_1)=0,
\ee
equality $
{\mathcal J}(t_1, u,m;\mu,z) = 0 $ holds if and only if there exists $v\in C^0(\T^d\times (t_1,T))$ such that the pair $(v,\mu)$ is a solution to the linearized problem \eqref{eq.LS} and 
$$
z=-\mu D_pH(x,Du)-mD^2_{pp}H(x,Du)Dv. 
$$ 
\end{Lemma}

\begin{proof} Let us first assume that $(v,\mu)$ is a solution to the linearized system \eqref{eq.LS} and let us set $z=-\mu D_pH(x,Du)-mD^2_{pp}H(x,Du)Dv$. Then $(\mu, z)$ satisfies \eqref{eq.cont.t1} and 
$$
\begin{array}{l}
\ds{\mathcal J}(t_1,u,m;\mu, z)\\ 
\qquad = \ds 
\int_{t_1}^T\inte m D^2_{qq}L\left(x, \frac{w}{m}\right) D^2_{pp}H(x,Du)Dv \cdot D^2_{pp}H(x,Du)Dv \\
\qquad \qquad  \ds + \int_{t_1}^T\inte \inte \frac{\delta f}{\delta m} (x,m(t),y)\mu(x,t)\mu(y,t) +
\inte \inte \frac{\delta g}{\delta m} (x,m(T),y)\mu(x,T)\mu(y,T).
\end{array}
$$
As $H$ is uniformly convex in the gradient variable and $w=-mD_pH(x,Du)$, we have 
\be\label{eq.D2LD2H}
D^2_{qq}L\left(x, \frac{w}{m}\right) D^2_{pp}H(x,Du) = I_d, 
\ee
so that 
$$
\begin{array}{l}
\ds{\mathcal J}(t_1,u,m;\mu, z)\\ 
\qquad = \ds 
\int_{t_1}^T\inte  m  D^2_{pp}H(x,Du)Dv\cdot Dv \\
\qquad  \ds + \int_{t_1}^T\inte \inte \frac{\delta f}{\delta m} (x,m(t),y)\mu(x,t)\mu(y,t) +
\inte \inte \frac{\delta g}{\delta m} (x,m(T),y)\mu(x,T)\mu(y,T) \; = \; 0,
\end{array}
$$
where the last equality is obtained by integrating the first equation in \eqref{eq.LS} multiplied by $\mu$ and adding it to the second one multiplied by $v$.\\

Conversely let us assume that ${\mathcal J}(t_1,u,m;\mu,z) = 0 $ holds for some $(\mu,z)$ satisfying \eqref{eq.cont.t1}. Let $v$ be the (continuous) solution to 
$$
\left\{\begin{array}{l}
\ds -\partial_t v -\Delta v +D_pH(x,Du)\cdot Dv=\frac{\delta f}{\delta m}(x,m)(\mu) \; {\rm in}\; \T^d\times (t_1,T), \\
\ds v(x,T)=\frac{\delta g}{\delta m}(x,m(T))(\mu(T))  \; {\rm in}\; \T^d.
\end{array}\right.
$$
We first claim  
that $(\mu,z)$ is a minimum point of 
${\mathcal J}(t_1,u,m; \cdot,\cdot)$ under the constraint \eqref{eq.cont.t1}. Indeed, by dynamic programming, $(m,w)$ is optimal for $J(t_1, m(t_1); \cdot, \cdot)$. So 
${\mathcal J}(t_1,u,m; \mu',z')\geq 0= {\mathcal J}(t_1,u,m; \mu,z)$ for any $(\mu',z')\in L^2$ such that \eqref{eq.cont.t1} holds, which proves the claim. 

Next we prove that $(v,\mu)$ is a solution to the linearized system \eqref{eq.LS}. As $(\mu,z)$ is a minimum point of the quadratic map
${\mathcal J}(t_1,u,m; \cdot,\cdot)$ under the constraint \eqref{eq.cont.t1}, we have by first order necessary condition and  for any $(\mu',z')$ such that \eqref{eq.cont.t1} holds, 
$$
\begin{array}{l}
\ds 
\int_{t_1}^T\inte   m^{-1} D^2_{qq}L\left(x, \frac{w}{m}\right) (z+  \mu D_pH(x,Du)) \cdot (z'+  \mu' D_pH(x,Du)) \\
\qquad  \ds + \int_{t_1}^T\inte \inte \frac{\delta f}{\delta m} (x,m(t),y)\mu(x,t)\mu'(y,t) +
\inte \inte \frac{\delta g}{\delta m} (x,m(T),y)\mu(x,T)\mu'(y,T) \; = \;0.
\end{array}
$$
(we used \eqref{rel.sym} in the above equality). Computing as usual  $\frac{d}{dt} \inte v\mu'$ and integrating in time we find:
$$
\inte \mu'(T) \frac{\delta g}{\delta m}(x,m(T))(\mu(T))=\int_{t_1}^T \inte \mu'( D_pH(x,Du)\cdot Dv-\frac{\delta f}{\delta m}(x,m)(\mu))
+z'\cdot Dv .
$$
Recalling \eqref{rel.sym}, this implies that 
$$
\begin{array}{l}
\ds 
\int_{t_1}^T\inte   m^{-1} D^2_{qq}L\left(x, \frac{w}{m}\right) (z+  \mu D_pH(x,Du)) \cdot (z'+  \mu' D_pH(x,Du)) 
\\
\qquad \qquad  \ds 
+ \int_{t_1}^T\inte (z'+\mu' D_pH(x,Du))\cdot Dv \; = \;0.
\end{array}
$$
Let us note that, for any $\alpha\in L^2$, there exists a unique solution $\mu'\in L^2$ to 
$$
\partial_t \mu'-\Delta \mu' -\dive (\mu' D_pH(x,Du))+\dive(\alpha)= 0 \; {\rm in }\; \T^d\times (t_1,T), \qquad \mu(t_1)=0.
$$
Then the pair $(\mu', z')$, with $z':= -\mu' D_pH(x,Du)+\alpha$, satisfies \eqref{eq.cont.t1}, so that we have 
$$
\begin{array}{l}
\ds 
\int_{t_1}^T\inte   \left(m^{-1} D^2_{qq}L\left(x, \frac{w}{m}\right) (z+  \mu D_pH(x,Du))+Dv\right) \cdot \alpha \; =\; 0
\end{array}
$$
for any $\alpha\in L^2$. Therefore 
$$
m^{-1} D^2_{qq}L\left(x, \frac{w}{m}\right) (z+  \mu D_pH(x,Du))+Dv=0,
$$
that, thanks to \eqref{eq.D2LD2H}, can be rewritten as 
$$
z= -\mu D_pH(x,Du)-m D^2_{pp}H(x,Du)Dv.
$$
This shows that the pair $(v,\mu)$ is a solution to the linearized system \eqref{eq.LS}. \\
\end{proof}

\begin{proof}[Proof of Theorem \ref{theo:stable}] Let $(v,\mu)$ be a solution of the linearized system \eqref{eq.LS}. From Lemma \ref{lem:equiv}, one has
${\mathcal J}(t_1,u,m; \mu,z)=0$, where $z= -\mu D_pH(x,Du)-mD^2_{pp}H(x,Du)Dv$. Let us extend $(\mu,z)$ to $\T^d\times [t_0,T]$ by setting $(\mu,z)=0$ on $\T^d\times [t_0,t_1]$.  Then $(\mu,z)$ satisfies \eqref{eq.cont} and ${\mathcal J}(t_0,u,m; \mu,z)=0$. Using again Lemma \ref{lem:equiv}, we know that there exists $\tilde v$ such that 
$(\tilde v, \mu)$ is a solution to the linearized system \eqref{eq.LS} on $\T^d\times (t_0,T)$ with initial condition $\mu(t_0)=0$ and $z= -\mu D_pH(x,Du)-mD^2_{pp}H(x,Du)D\tilde v$. By parabolic regularity, as $(\tilde v, \mu)$ solves  \eqref{eq.LS}, it  satisfies $\tilde v\in C^1$ and $\mu\in C^0$. Thus $z(t_1)=0=  -\mu(t_1) D_pH(x,Du(t_1))-m(t_1)D^2_{pp}H(x,Du(t_1))D\tilde v(t_1)$. Now recalling that $\mu(t_1)=0$, $m(t_1)>0$ and $D^2_{pp}H>0$, this implies that $D\tilde v(t_1)=0$.  

So the pair $(D\tilde v, \mu)$ is a classical solution to 
$$
\left\{\begin{array}{l}
\ds -\partial_t (\partial_{x_i}\tilde v) -\Delta (\partial_{x_i}\tilde v) +g_i= 0\; {\rm in}\; \T^d\times (t_1,T),\; i=1, \dots, d, \\
\ds \partial_t \mu -\Delta \mu + \dive(h)= 0 \; {\rm in}\; \T^d\times (t_1,T) \\
\ds \mu(x,t_1)=0, \; D\tilde v(x,t_1)=0 \; {\rm in}\; \T^d.
\end{array}\right.
$$
where 
$$
g_i(x,t)= \partial_{x_i}\left(D_pH(x,Du)\cdot D\tilde v-\frac{\delta f}{\delta m}(x,m)(\mu) \right)
$$
and 
$$
h(x,t)= -\mu D_pH(x, Du)-mD^2_{pp}H(x,Du)D\tilde v.
$$
Note that 
$$
\sum_{i=1}^d |g_i(x,t)|^2 \leq C\left( |D\tilde v|^2 + |D^2\tilde v|^2 + \|\mu(t)\|_{L^2}^2\right)
$$
while 
$$
|h(x,t)|^2\leq C\left(|\mu(x,t)|^2 + |D\tilde v(x,t)|^2\right)
$$
and 
$$
|\dive (h)(x,t)|^2\leq C\left(|D\tilde v|^2 + |D^2\tilde v|^2 + |\mu(x,t)|^2+|D\mu(x,t)|^2\right).
$$
Then Theorem \ref{theo:continuation} in  the Appendix states that $D\tilde v=0$ and $\mu=0$. As $(v,\mu)$ solves \eqref{eq.LS}, this also implies that $v=0$. 
\end{proof}

\section{Application to a learning procedure}\label{sec:fictitious}

In order to illustrate the notion of stable solution in potential games, we consider a learning procedure and show that this procedure converges if one starts from a neighborhood of a stable MFG equilibrium.

In \cite{CH} the following algorithm (inspired by the Fictitious Play \cite{BrownFictitiousPlay}) was introduced: let $\mu^0\in C^0([0,T],\Pk)$ and define by induction $(u^n, m^n)$ the solution to 
\be\label{eq:Induc}
\left\{ \begin{array}{l}
-\partial_t u^{n+1}-\Delta u^{n+1}+H(x, Du^{n+1})=f(x, \mu^n))\\
\partial_t m^{n+1}-\Delta m^{n+1}-\dive ( m^{n+1}D_PH(x,Du^{n+1}))=0\\
m^{n+1}(0)= m_0, \qquad u^{n+1}(T,x)= g(x,\mu^n(T))
\end{array}\right.
\ee
and
$$
\mu^{n+1}= \frac{n}{n+1} \mu^n+ \frac{1}{n+1} m^{n+1}.
$$

Following ideas of Monderer and Shapley \cite{MondrerShapley2}, it was proved in \cite{CH} that the family $(u^n,m^n)$ is compact in $C^{1,0}\times C^0$ and that any converging subsequence is a solution of the MFG system. 

Our main result is the convergence of $(u^n,m^n)$ to a stable solution $(u,m)$ as soon as $\mu^0$ is sufficiently close to $m$. This can be understood as a robustness property of the equilibrium $(u,m)$.

\begin{Theorem}\label{theo:CvLearning} Assume that the game is potential, i.e., satisfies \eqref{eq:potential} for some $C^1$ maps $F,G:\Pk\to\R$.  Let $(u,m)$ be a solution to the MFG system \eqref{eq:MFG} starting from $(0,m_0)$ with $m_0\in \Pk$ and assume that the solution is stable. Then there exists $\delta>0$ such that, if $\ds \sup_{t\in [0,T]} \dk(m(t),\mu^0(t))\leq \delta$, the sequence $(u^n,m^n)$ defined above converges to $(u,m)$ as $n\to+\infty$ in $C^{1,0}\times C^0$. 
\end{Theorem}

In other worlds, stable equilibria are local attractors for the learning procedure. Note also that, in contrast to \cite{CH}, we show here that the full sequence  $(u^n,m^n)$ converges, although the MFG system may have several solutions. \\

We now start the proof of Theorem \ref{theo:CvLearning}. Let us recall that, by the stability assumption of the MFG equilibrium $(u,m)$, the linearized system 
\be\label{LinSystBis}
\left\{\begin{array}{l}
-\partial_t v -\Delta v +D_pH(x,Du)\cdot Dv=\frac{\delta f}{\delta m}(x,m)(\rho) \; {\rm in}\; \T^d\times (0,T), \\
\partial_t \rho -\Delta \rho - \dive(\rho D_pH(x, Du))-\dive (mD^2_{pp}H(x,Du)Dv)= 0 \; {\rm in}\; \T^d\times (0,T) \\
\rho(x,0)=0, \; v(x,T)=\frac{\delta g}{\delta m}(x,m(T))(\rho(T))  \; {\rm in}\; \T^d.
\end{array}\right.
\ee
has the pair $(v,\rho)=(0,0)$ as unique solution. As a consequence, we have:
\begin{Lemma}\label{lem:boundLS} There is a constant $C>0$ such that, for any $a,b\in C^0(\T^d\times [0,T])$, $c\in C^0(\T^d)$ and $(v,\rho)$ solution to 
\be\label{eq:PerturbedLS}
\left\{\begin{array}{l}
-\partial_t v -\Delta v +D_pH(x,Du)\cdot Dv=\frac{\delta f}{\delta m}(x,m)(\rho) +a(x,t) \; {\rm in}\; \T^d\times (0,T), \\
\partial_t \rho -\Delta \rho - \dive(\rho D_pH(x, Du))-\dive (mD^2_{pp}H(x,Du)Dv)= \dive(b(x,t)) \; {\rm in}\; \T^d\times (0,T) \\
\rho(x,0)=0, \; v(x,T)=\frac{\delta g}{\delta m}(x,m(T))(\rho(T)) +c \; {\rm in}\; \T^d,
\end{array}\right.
\ee
one has
$$
\|v\|_{C^{1,0}}+ \|\rho\|_{C^0}\leq C\left(\|a\|_{C^0}+\|b\|_{C^0}+\|c\|_{C^0}\right).
$$
\end{Lemma}

\begin{proof} We argue by contradiction. Since the system is affine, this means that  there exists $a^n, b^n,c^n$ and $(v^n, \rho^n)$ solution to \eqref{eq:PerturbedLS} with 
$$
\|a^n\|_{C^0}+\|b^n\|_{C^0}+\|c^n\|_{C^0}\leq 1 \qquad {\rm and}\qquad \theta^n:=\|v^n\|_{C^{1,0}}+ \|\rho^n\|_{C^0} \geq n.
$$
Let us set 
$$
\tilde v^n :=\frac{v^n}{\theta^n}, \qquad \tilde \rho^n := \frac{\rho^n}{\theta^n}. 
$$
The pair $(\tilde v^n , \tilde \rho^n)$ solves 
$$
\left\{\begin{array}{l}
-\partial_t \tilde v^n -\Delta \tilde v^n +D_pH(x,Du)\cdot D \tilde v^n=\frac{\delta f}{\delta m}(x,m)(\tilde\rho^n) +a^n(x,t)/\theta^n \; {\rm in}\; \T^d\times (0,T), \\
\partial_t \tilde\rho^n -\Delta \tilde\rho^n - \dive(\tilde\rho^n D_pH(x, Du))-\dive (mD^2_{pp}H(x,Du)D \tilde v^n)= \dive\left(b^n(x,t)/\theta^n\right) \; {\rm in}\; \T^d\times (0,T) \\
\tilde\rho^n(x,0)=0, \; \tilde v^n(x,T)=\frac{\delta g}{\delta m}(x,m(T))(\tilde \rho(T)) +c^n/\theta^n \; {\rm in}\; \T^d.
\end{array}\right.
$$
The $\tilde v_n$ satisfy a linear parabolic equation with bounded coefficients. Therefore  (see for instance, Theorem 9.1 and Theorem 10.1 in \cite{LSU}), the sequences  $(\tilde v_n)$ and $(D \tilde v_n)$ are bounded in $C^{\alpha,\alpha/2}$ for some $\alpha\in(0,1)$. In the same way, the $\tilde\rho^n$ satisfy a linear parabolic equation in divergence form with bounded coefficients, so that the sequence $(\tilde\rho^n)$ is bounded in $C^{\alpha,\alpha/2}$. Thus, up to a subsequence $(\tilde v_n)$ and $(D \tilde v_n)$ converge uniformly to some $v$ and $Dv$ while  $(\tilde\rho^n)$ converges to some $\rho$. Note that, by definition of $\theta^n$ and $(\tilde v_n,\tilde \rho_n)$, we have
$$
\|v\|_{C^{1,0}}+\|\rho\|_{C^0}=\lim_{n \rightarrow +\infty} (  \|\tilde v^n\|_{C^{1,0}}+\|\tilde \rho^n\|_{C^0}  )=1.
$$
On the other hand, the pair $(v,\rho)$ is a weak solution to the linearized system \eqref{LinSystBis}, so that, by our assumption on this system, $(v,\rho)$ must actually vanish. So there is a contradiction. 
\end{proof}
 
\begin{Lemma}\label{lem:CondIneq} There exists $\eta>0$ and $C>0$ such that, if 
$$
\|D(u^{n+1}-u)\|_{C^0}+\|m^{n+1}-m\|_{C^0}\leq \eta,
$$
then 
$$
\|u^{n+1}-u\|_{C^{1,0}}+\|m^{n+1}-m\|_{C^0}\leq C \|m^{n+1}-\mu^n\|_{C^0}.
$$
\end{Lemma}

\begin{proof} We rewrite system \eqref{eq:Induc} as \eqref{eq:PerturbedLS} with $v=u^{n+1}-u$, $\rho=m^{n+1}-m$,
$$
\begin{array}{rl}
\ds a \; := & \ds H(x,Du^{n+1})-H(x,Du)- D_pH(x, Du) \cdot D(u^{n+1}-u) \\
& \ds + f(x,\mu^{n}(t))-f(x,m(t))- \frac{\delta f}{\delta m}(m(t))(m^{n+1}(t)-m(t)) ,
\end{array}
$$
$$
\begin{array}{rl}
\ds b\; := & \ds  m^{n+1}D_pH(x,Du^{n+1})-mD_pH(x,Du)- (m^{n+1}-m)D_pH(x,Du)\\
& \ds \qquad -mD^2_{pp}H(x,Du)D(u^{n+1}-u)
\end{array}
$$
and
$$
c(x):=  g(x,\mu^{n}(T))-g(x,m(T))- \frac{\delta g}{\delta m}(m(t))(m^{n+1}(T)-m(T)).
$$
Then, by  Lemma \ref{lem:boundLS}, we get 
$$
\begin{array}{l}
\ds \|u^{n+1}-u\|_{C^{1,0}}+\|m^{n+1}-m\|_{C^0}\\
\qquad  \leq  \ds C (\|a\|_{C^0}+\|b\|_{C^0}+\|c\|_{C^0}) \\
\qquad  \leq  \ds C\left(\|D(u^{n+1}-u)\|_{C^0}^2+\|\mu^{n}-m\|_{C^0}^2+\|m^{n+1}-\mu^n\|_{C^0}+\|m^{n+1}-m\|^2_{C^0}\right) \\
\qquad  \leq  \ds C\left(\|D(u^{n+1}-u)\|_{C^0}^2+\|m^{n+1}-m\|^2_{C^0}+\|m^{n+1}-\mu^n\|_{C^0}\right)
\end{array}
$$
where the constant $C$ in the last inequality depends only on the data. Thus, if $\eta>0$ is small enough and  $\|D(u^{n+1}-u)\|_{C^0}+\|m^{n+1}-m\|_{C^0}$ is not larger than $\eta$, one can absorb the square terms of the right-hand side into the left-hand side to get 
the conclusion.
\end{proof}

\begin{proof}[Proof of Theorem \ref{theo:CvLearning}] Following \cite{CH}, the difference $(m^{n+1}-\mu^n)$ converges to $0$ in $C^0$: so we can find $N\geq 0$ large enough so that 
\be\label{eq:keystepCH}
\|m^{n+1}-\mu^n\|_{C^0}\leq \eta/(2C)\qquad \forall n\geq N, 
\ee
where $\eta$ and $C$ are given by Lemma \ref{lem:CondIneq}. 

Next we note that the $u^n$ are uniformly Lipschitz continuous. So the $m^n$ solve a parabolic equation in divergence form with bounded coefficient and therefore are bounded in $C^{\alpha, \alpha/2}$. Thus so are the $\mu^n$. Moreover, by definition of $\mu^n$, we have that
$$
\|\mu^{n+1}-\mu^n\|_{C^0}= \frac{1}{n}\|m^{n+1}-\mu^n\|_{C^0}\leq \frac{C}{n}.
$$
As $u^{n+2}$ and $u^{n+1}$ solve the same equation with right-hand side given respectively by $f(x,\mu^n)$ and $f(x,\mu^{n+1})$ and terminal condition given by $g(x,\mu^n(T))$ and $g(x,\mu^{n+1}(T))$ respectively, we conclude that 
\be\label{eq:un+2-un+1}
\|u^{n+2}-u^{n+1}\|_{C^{1,0}}\leq C \left(\|f(\cdot,\mu^n)-f(\cdot,\mu^{n+1})\|_{C^{1,0}}+\|g(\cdot,\mu^n(T))-g(\cdot,\mu^{n+1}(T))\|_{C^{1,0}}\right) \leq \frac{C}{n}.
\ee
Plugging this estimate into the equation satisfied by $m^{n+1}-m^{n+2}$ also gives
\be\label{eq:mn+2-mn+1}
\|m^{n+1}-m^{n+2}\|_{C^0}\leq  \frac{C}{n}.
\ee
From now on we assume that $N$ is so large that the right-hand sides $C/n$ in \eqref{eq:un+2-un+1} and \eqref{eq:mn+2-mn+1} are less than $\eta/4$ for any $n\geq N$. 

Next we note that, if $\mu^0$ is sufficiently close to $m$, then $u^1$ is close to $u$ (in $C^{1,0}$) and $m^1$ is close to $m$ (in $C^0$). This in turn shows that $\mu^1$ is close to $m$ (in $C^0$). By induction, we obtain that, for any fixed $n$, one can choose $\mu^0$ sufficiently close to $m$ so that 
$u^{n+1}$ is close to $u$ (in $C^{1,0}$) and $m^{n+1}$ is close to $m$ (in $C^0$). Applying this argument to the integer $N$ defined above, we obtain that one can choose $\mu^0$ so close to $m$ that 
$$
\|u^{N+1}-u\|_{C^{1,0}}+\|m^{N+1}-m\|_{C^0}\leq \eta.
$$
We claim that this inequality propagate to any $n+1\geq N+1$. Indeed, let us assume that it holds for some $n+1$. Then, by Lemma \ref{lem:CondIneq}, we have 
$$
\|u^{n+1}-u\|_{C^{1,0}}+\|m^{n+1}-m\|_{C^0}\leq C\|m^{n+1}-\mu^n\|_{C^0}\leq \eta/2 
$$
thanks to \eqref{eq:keystepCH}. So, by \eqref{eq:un+2-un+1} and \eqref{eq:mn+2-mn+1} and the choice of $N$, we also have 
$$
\begin{array}{l}
\ds \|u^{n+2}-u\|_{C^{1,0}}+\|m^{n+2}-m\|_{C^0}\\
\ds \qquad \leq \|u^{n+1}-u\|_{C^{1,0}}+\|m^{n+1}-m\|_{C^0}+  \|u^{n+2}-u^{n+1}\|_{C^{1,0}}+\|m^{n+1}-m^{n+2}\|_{C^0} \leq \eta,
\end{array}
$$
which proves the claim. 

In particular, we can apply Lemma \ref{lem:CondIneq} to any $n\geq N$ to get
$$
\|u^{n+1}-u\|_{C^{1,0}}+\|m^{n+1}-m\|_{C^0}\leq C\|m^{n+1}-\mu^n\|_{C^0},
$$
where the right-hand side tends to $0$ as $n\to +\infty$ by Theorem 2.1  in \cite{CH}. This proves the convergence of $(u^n)$ and $(m^n)$ to $u$ and $m$ respectively. 
\end{proof}

\section{Appendix}

Let $(u,\mu)=((u_i)_{i=1, \dots, d},\mu)$ be a classical solutions to the forward-backward system 
$$
\left\{\begin{array}{l}
-\partial_t u_i -\Delta u_i +g_i(x,t)=0 \; {\rm in}\; \T^d\times (0,T), \; i=1, \dots, d,\\
\partial_t \mu -\Delta \mu + \dive(h(x,t))= 0 \; {\rm in}\; \T^d\times (0,T)
\end{array}\right.
$$
We suppose that
$$
 u_i(x,0)=\mu(x,0)=0 \; {\rm in}\; \T^d, \; i=1, \dots, d, 
 $$
$$
\sum_{i=1}^d \left| g_i(x,t)\right|^2 \leq C_0 \left( \|\mu(\cdot,t)\|^2_2+|u(x,t)|^2+|Du(x,t)|^2\right),
$$
$$
|h(x,t)|^2\leq C_0\left( |\mu(x,t)|^2+|u(x,t)|^2\right), 
$$
$$
|\dive (h)(x,t)|^2 \leq C_0  \left( |\mu(x,t)|^2+|D\mu(x,t)|^2+|u(x,t)|^2+|Du(x,t)|^2\right).
$$

\begin{Theorem}\label{theo:continuation} Under the above assumptions, one has $u(x,t)=0$ and $\mu(x,t)=0$ on $\T^d\times [0,T]$. 
\end{Theorem}

 The argument is standard and goes back to J.L. Lions and B. Malgrange in  \cite{LM} (see also Cannarsa-Tessitore \cite{CT} for a forward-backward system). 

\begin{proof} It is enough to argue for $T>0$ sufficiently small and prove that $\mu=0$ and $u=0$ on $[0,T/2]$. Let $\theta:[0,T]\to [0,1]$ be a smooth, non increasing function with $\theta(t)=1$ in $[0,T/2]$, $\theta (t)=0$ in $[2T/3, T]$ and $\|\theta'(t)\|_\infty\leq C/T$. For $\kappa\geq 1$ we set 
$$
\tilde u_i(x,t)= e^{\kappa(t-T)^2/2}\theta(t) u_i(x,t), \qquad \tilde \mu(x,t)= e^{\kappa(t-T)^2/2}\theta(t) \mu(x,t).
$$
Then $(\tilde u, \tilde \mu)= ((\tilde u_i), \tilde \mu)$ satisfies 
\be\label{ljhbzcsd}
\left\{\begin{array}{ll}
(i) & -\partial_t \tilde u_i -\Delta \tilde u_i +\kappa (t-T)\tilde u_i +e^{\kappa(t-T)^2/2}\theta' u_i +e^{\kappa(t-T)^2/2}\theta g_i=0, \\
(ii) & \partial_t \tilde \mu -\Delta  \tilde \mu - \kappa (t-T)\tilde \mu- e^{\kappa(t-T)^2/2}\theta' \mu + e^{\kappa(t-T)^2/2}\theta\dive( h)= 0,  \\
(iii) & \tilde u_i(x,0)=\tilde \mu(x,0)=\tilde u_i(x,T)=\tilde \mu(x,T)=0.
\end{array}\right.
\ee

Multiplying \eqref{ljhbzcsd}-(ii) by $\partial_t \tilde \mu$ and integrating in time-space, we obtain, after integration by parts the second term: 
$$
\int_0^T \inte (\partial_t \tilde \mu)^2 +\frac12 \partial_t |D\tilde \mu|^2 - \frac{\kappa}{2} (t-T)\partial_t(\tilde \mu)^2  - \partial_t \tilde \mu\left(e^{\kappa(t-T)^2/2}\theta' \mu -e^{\kappa(t-T)^2/2}\theta\dive( h)\right)= 0.
$$
We integrate the second term in time and integrate  by parts in time the third one: taking into account the fact that $\tilde \mu(\cdot,0)=\tilde \mu(\cdot,T)=0$, we get: 
$$
\int_0^T \inte (\partial_t \tilde \mu)^2  + \frac{\kappa}{2} (\tilde \mu)^2  - \partial_t \tilde \mu\left(e^{\kappa(t-T)^2/2}\theta' \mu -e^{\kappa(t-T)^2/2}\theta\dive( h)\right)= 0.
$$
By Young's inequality we obtain 
\be\label{deuxdeux}
\int_0^T \inte \frac12(\partial_t \tilde \mu)^2  + \frac{\kappa}{2} (\tilde \mu)^2  \leq C\int_0^T \inte
\left(e^{\kappa(t-T)^2}(\theta')^2 \mu^2 +e^{\kappa(t-T)^2}\theta^2|\dive( h)|^2\right).
\ee
We argue in the same way for $u_i$: we multiply \eqref{ljhbzcsd}-(i) by $\partial_t \tilde u_i$ and integrate in space-time: 
$$
\int_0^T\inte -(\partial_t \tilde u_i)^2 +\frac12 \partial_t |D \tilde u_i|^2  +\frac{\kappa}{2} (t-T)\partial_t (\tilde u_i)^2 +\partial_t \tilde u_i\left( e^{\kappa(t-T)^2/2}\theta' u_i +e^{\kappa(t-T)^2/2}\theta g_i\right)=0.
$$
This yields to
\be\label{troistrois}
\int_0^T\inte \frac12 (\partial_t \tilde u_i)^2 +\frac{\kappa}{2} (\tilde u_i)^2  \leq 
C \int_0^T\inte
\left( e^{\kappa(t-T)^2}(\theta')^2 u_i^2 +e^{\kappa(t-T)^2}\theta^2 g_i^2\right).
\ee

Next we multiply \eqref{ljhbzcsd}-(ii) by $\tilde \mu$ and integrate in time-space: using as usual that $\tilde \mu(\cdot,0)=\tilde \mu(\cdot,T)=0$, we obtain
$$
\int_0^T\inte |D\tilde \mu|^2 - \kappa (t-T)(\tilde \mu)^2 +\tilde \mu( - e^{\kappa(t-T)^2/2}\theta' \mu) - D\tilde \mu\cdot (e^{\kappa(t-T)^2/2}\theta h)= 0.
$$
Hence, by Young's inequality,  
\be\label{quatquat}
\int_0^T\inte \frac12 |D\tilde \mu|^2 \leq  \int_0^T\inte  \tilde \mu^2 + C \left(e^{\kappa(t-T)^2}(\theta')^2 \mu^2+ e^{\kappa(t-T)^2}\theta^2 h^2\right).
\ee
In the same way, we multiply \eqref{ljhbzcsd}-(i) by $u_i$ and integrate in time-space: 
$$
\int_0^T\inte |D \tilde u_i|^2 +\kappa (t-T)(\tilde u_i)^2 +\tilde u_i\left( e^{\kappa(t-T)^2/2}\theta' u_i +e^{\kappa(t-T)^2/2}\theta g_i\right)=0.
$$
Then, for $\ep\in (0,1)$ to be chosen later,
\be\label{cinqcinq}
\int_0^T\inte |D \tilde u_i|^2 \leq  \int_0^T\inte (\kappa T+\ep^{-1}) (\tilde u_i)^2 +C\ep\left( e^{\kappa(t-T)^2}(\theta')^2 u_i^2 +e^{\kappa(t-T)^2}\theta^2 g_i^2\right).
\ee

We now take into account the assumptions on the $g_i$, $h$ and $\dive(h)$:   \eqref{deuxdeux} becomes
$$
\begin{array}{l}
\ds \frac{\kappa}{2}  \int_0^T  \|\tilde \mu(t)\|^2_2 \\
\ds \qquad  \leq C\int_0^T 
\left(e^{\kappa(t-T)^2}(\theta')^2 \|\mu(t)\|^2_2 + e^{\kappa(t-T)^2}\theta^2\left( \|u(t)\|^2_2+\|Du(t)\|^2_2+ \|\mu(t)\|^2_2+ \|D\mu(t)\|^2_2 \right)\right) .
\end{array}
$$
Rearranging we obtain 
\be\label{deuxdeuxbis}
\ds \kappa  \int_0^T  \|\tilde \mu(t)\|^2_2  \leq C\int_0^T 
\left(e^{\kappa(t-T)^2}(\theta')^2 \|\mu(t)\|^2_2 + \|\tilde u(t)\|^2_{H^1}+ \|\tilde \mu(t)\|^2_{H^1} \right) .
\ee
 Note that we used the notation $\displaystyle   \|v(t)\|^2_{H^1}:=\|v(t)\|^2_2+ \|Dv(t)\|^2_2 $. 
The same argument for \eqref{troistrois} yields to
\be\label{troistroisbis}
\kappa \int_0^T  \|\tilde u_i(t)\|^2_2  \leq 
C \int_0^T  e^{\kappa(t-T)^2}(\theta')^2 \|u_i(t)\|_2^2 + \|\tilde u(t)\|^2_{H^1}+ \|\tilde \mu(t)\|^2_{H^1},
\ee
while \eqref{quatquat} and \eqref{cinqcinq} become respectively: 
\be\label{sixsixbis}
\int_0^T \|D\tilde \mu(t)\|_2^2 \leq  C \int_0^T  e^{\kappa(t-T)^2}(\theta')^2 \|\mu(t)\|_2^2+ \|\tilde u(t)\|^2_2+\|\tilde \mu(t)\|_2^2,
\ee
and 
\be\label{septsept}
\int_0^T \|D \tilde u_i(t)\|_2^2 \leq  \int_0^T  (\kappa T+\ep^{-1}) \|\tilde u_i(t)\|_2^2 +C\ep\left( e^{\kappa(t-T)^2}(\theta')^2 \|u_i(t)\|_2^2 
+ \|\tilde u(t)\|^2_{H^1}+ \|\tilde \mu(t)\|^2_{H^1}  \right).
\ee
Summing \eqref{septsept} over $i$ yields, for $\ep$ small enough (depending only on the constant $C$ in \eqref{septsept}), 
$$
\int_0^T \|D \tilde u(t)\|_2^2 \leq  C \int_0^T  (\kappa T+1) \|\tilde u(t)\|_2^2 + e^{\kappa(t-T)^2}(\theta')^2 \|u(t)\|_2^2 + \|\tilde \mu(t)\|^2_{H^1}.
$$
Plugging \eqref{sixsixbis} into the above inequality gives: 
\be\label{septseptbis}
\int_0^T \|D \tilde u(t)\|_2^2 \leq  C \int_0^T  (\kappa T+1) \|\tilde u(t)\|_2^2+ \|\tilde \mu(t)\|_2^2+ e^{\kappa(t-T)^2}(\theta')^2 (\|u(t)\|_2^2 + \|\mu(t)\|^2_{2}).
\ee

Collecting \eqref{deuxdeuxbis}, \eqref{troistroisbis}  (summing on $i=1,\dots, d$), \eqref{sixsixbis} and \eqref{septseptbis}  yields to 
$$
\begin{array}{rl}
\ds \kappa \int_0^T \|\tilde \mu(t)\|^2_2+ \|\tilde u(t)\|^2_2 \; \leq  &
\ds C\int_0^T 
e^{\kappa(t-T)^2}(\theta')^2 (\|\mu(t)\|^2_2+ \|u(t)\|_2^2) + \|\tilde u(t)\|^2_{H^1}+ \|\tilde \mu(t)\|^2_{H^1} \\
\leq & \ds C\int_0^T 
e^{\kappa(t-T)^2}(\theta')^2 (\|\mu(t)\|^2_2+ \|u(t)\|_2^2) + (\kappa T+1)\|\tilde u(t)\|^2_2+ \|\tilde \mu(t)\|^2_2
\end{array}
$$
We can now fix $T>0$ small enough, so that, for any $\kappa$ large enough, 
$$
\frac{\kappa}{2} \int_0^T \|\tilde \mu(t)\|^2_2+ \|\tilde u(t)\|^2_2
\leq 
C\int_0^T 
e^{\kappa(t-T)^2}(\theta')^2 (\|\mu(t)\|^2_2+ \|u(t)\|_2^2).
$$
By the choice of $\theta$ (namely $\theta=1$ on $[0,T/2]$ and $\|\theta'\|_\infty\leq C/T$), this implies that 
$$
\frac{\kappa}{2} \int_0^{T/2} e^{\kappa(t-T)^2} \left(\|\mu(t)\|^2_2+ \| u(t)\|^2_2\right)
\leq 
\frac{C}{T}\int_{T/2}^T 
e^{\kappa(t-T)^2} (\|\mu(t)\|^2_2+ \|u(t)\|_2^2).
$$
Hence
$$
\frac{\kappa}{2} e^{\kappa(T/2)^2}  \int_0^{T/2}  \left(\|\mu(t)\|^2_2+ \| u(t)\|^2_2\right)
\leq 
\frac{C}{T}e^{\kappa(T/2)^2} \int_{T/2}^T  (\|\mu(t)\|^2_2+ \|u(t)\|_2^2).
$$
Dividing by  $e^{\kappa(T/2)^2} $ and letting $\kappa\to +\infty$ yields to $\mu=0$ and $u=0$ on $[0,T/2]$.

\end{proof}



\thebibliography{99}
\bibitem{AGS}  Ambrosio, L., Gigli, N., Savar\`e, G. Gradient flows in metric spaces and in the space of probability measures. Lectures in Mathematics ETH Z\"urich. Birkh\"auser Verlag, Basel, 2008.

\bibitem{BaFi} Bardi M., Fischer M. In preparation. 

\bibitem{BrownFictitiousPlay}
Brown, G. W. (1951). {\it Iterative solution of games by Fictitious Play.} Activity analysis of production and allocation, 13(1), 374-376.

\bibitem{CannarsaSinestrari}
Cannarsa, P., Sinestrari, C.
Semiconcave functions, Hamilton-Jacobi equations and optimal control.
 Birkh\"auser, Boston, 2004.

\bibitem{CT} P. Cannarsa, M. E. Tessitore {\it  Optimality conditions for boundary control problems of parabolic type}. Control and estimation of distributed parameter systems: nonlinear phenomena (Vorau, 1993), 79-96, Internat. Ser. Numer. Math., 118, Birkh\"auser, Basel, 1994. 


\bibitem{CGPT} Cardaliaguet P, Graber J., Porretta A., Tonon D., {\it Second order mean field games with degenerate diffusion and local coupling.} Nonlinear Differ. Equ. Appl. 22, 5, (2015), 1287-1317.

\bibitem{CH} Cardaliaguet P., Hadikhanloo S. {\it  Learning in Mean Field Games: the Fictitious Play.} To appear in COCV. Pre-print hal-01179503.  

\bibitem{CDLL} Cardaliaguet P., Delarue F., Lasry J.M., Lions P.L. {\it The master equation and the convergence problem in mean field games.} Pre-print arXiv:1509.02505

\bibitem{HCMieeeAC06} Huang, M., Malham\'e, R.P., Caines, P.E.  (2006). 
{\it Large population stochastic dynamic games: closed-loop McKean-Vlasov systems and the Nash certainty equivalence
principle.} Communication in information and systems. Vol. 6, No. 3, pp. 221-252. 


\bibitem{LSU}
Lady\v{z}enskaja O.A., Solonnikov V.A and Ural'ceva N.N 
 Linear and quasilinear equations of parabolic type.
 Translations of Mathematical Monographs, Vol. 23 American Mathematical Society, Providence, R.I. 1967


\bibitem{LL06cr1} Lasry, J.-M., Lions, P.-L. {\it Jeux \`a champ moyen. I. Le cas stationnaire.}
C. R. Math. Acad. Sci. Paris  343  (2006),  no. 9, 619-625.

\bibitem{LL06cr2} Lasry, J.-M., Lions, P.-L. {\it Jeux \`a champ moyen. II. Horizon fini et contr\^ole optimal.}
C. R. Math. Acad. Sci. Paris  343  (2006),  no. 10, 679-684.

\bibitem{LL07mf} Lasry, J.-M., Lions, P.-L. {\it Mean field games.}  Jpn. J. Math.  2  (2007),  no. 1, 229--260.


\bibitem{LM} Lions J. L., Malgrange B., {\it Sur l'unicit\'e r\'etrograde dans les probl\`emes mixtes paraboliques.} (French) Math. Scand. 8 1960 277-286.


\bibitem{MondrerShapley2}
Monderer D., and Shapley L.S., {\it Fictitious play property for games with identical interests.} Journal of economic theory 68.1 (1996): 258-265.
\end{document}